\documentclass[a4paper,twopage,reqno,12pt]{amsart}
\usepackage[top=30mm,right=30mm,bottom=30mm,left=30mm]{geometry}
\usepackage{graphicx}
\usepackage{amsmath}
\usepackage{amssymb}
\usepackage{amsfonts}
\usepackage{amsthm}
\usepackage{amstext}
\usepackage{amsbsy}
\usepackage{amsopn}
\usepackage{amscd}
\usepackage{enumerate}
\usepackage{multirow}
\usepackage{float}
\usepackage{color}
\usepackage{longtable}
\usepackage[colorlinks]{hyperref}
\usepackage[hyperpageref]{backref}
\usepackage{lscape}
\usepackage{longtable}
\newtheorem{theorem}{Theorem}[section]

\newtheorem{lemma}[theorem]{Lemma}
\newtheorem{proposition}[theorem]{Proposition}

\theoremstyle{definition}

\newtheorem{conjecture}{Conjecture}[section]
\theoremstyle{remark}

\numberwithin{equation}{section}

\providecommand{\Out}{\mathop{\rm Out}\nolimits}%
\providecommand{\cd}{\mathop{\rm cd}\nolimits}%

\newcommand{\Irr}{\textsf{Irr}}
\newcommand{\Aut}{\textsf{Aut}}

\renewcommand{\leq}{\leqslant}

\newcommand{\Zbb}{\mathbb{Z}}

\providecommand{\mod}[1]{ \ (\mathrm{mod}\ #1)}

\begin{document}

\title[Groups with character degrees of almost simple Mathieu groups]{On groups with the same character degrees as almost simple groups with socle the Mathieu groups}

\author[S.H. Alavi]{Seyed Hassan Alavi}
 \address{
 S.H. Alavi, Department of Mathematics, Faculty of Science,
 Bu-Ali Sina University, Hamedan, Iran}
 \email{alavi.s.hassan@basu.ac.ir alavi.s.hassan@gmail.com (Gmail is preferred)}

\author[A. Daneshkhah]{Ashraf Daneshkhah}
 \thanks{Corresponding author: Ashraf Daneshkhah}
 \address{%
 A. Daneshkhah, Department of Mathematics,
 Faculty of Science,
 Bu-Ali Sina University, Hamedan, Iran}
 \email{adanesh@basu.ac.ir daneshkhah.ashraf@gmail.com (Gmail is preferred)}

\author[A. Jafari]{Ali Jafari}
 \address{%
 A. Jafari, Department of Mathematics,
 Faculty of Science,
 Bu-Ali Sina University, Hamedan, Iran}
 \email{a.jaefary@gmail.com}



\begin{abstract}
    Let $G$ be a finite group and $\cd(G)$ denote the set of complex irreducible character degrees of $G$. In this paper, we prove that if $G$ is a finite group and $H$ is an almost simple group whose socle is Mathieu group such that $\cd(G) =\cd(H)$, then there exists an Abelian subgroup $A$ of $G$ such that $G/A$ is isomorphic to $H$. This study is heading towards the study of an extension of Huppert's conjecture (2000) for almost simple groups.
\end{abstract}

\date{\today}
\maketitle


\section{Introduction}\label{sec:Intro}

Let $G$ be a finite group, and let $\Irr(G)$ be the set of all complex irreducible character degrees of $G$. Denote by $\cd(G)$ the set of character degrees of $G$, that is to say, $\cd(G)=\{\chi(1)\vert \chi \in \Irr(G)\}$.
It is well-known that the complex  group algebra $\mathbb{C}G$ determines character degrees and their multiplicities.

There is growing interest in a question regarding the structure of $G$ which can be recovered from the character degree set of $G$ with or without multiplicity. It is well-known that the character degree set of $G$ can not use to completely determine the structure of $G$. For example, the non-isomorphic groups $D_8$ and $Q_8$ not only have the same set of character degrees, but also share the same character table. The character degree set cannot be used to distinguish between solvable and nilpotent groups. For example, if $G$ is either $Q_{8}$ or $S_{3}$, then $\cd(G) = \{1, 2\}$. Recently, Navarro \cite{Navarro15:cdsol} showed that the character degree set alone cannot determine the solvability of the group. Indeed, he constructed a finite perfect group $H$ and a finite solvable group $G$ such that $\cd(G) = \cd(H)$. It is also discovered by Navarro and Rizo \cite{Novarro14:cdnil} that there exists a finite perfect group and a finite nilpotent group with the same character degree set. Notice that in both examples, these finite perfect groups are not nonabelian simple. It remains open whether the complex group algebra can determine the solvability of the group or not (see Brauer's Problem 2 \cite{Brauer63}).

However, the situation for simple groups and related groups is rather different. It has been proved recently that all \emph{quasisimple} groups are uniquely determined up to isomorphism by their complex group algebras~\cite{Bessenrodt15:doublecover} in which a finite group $G$ is quasisimple if $G$ is perfect and $G/Z(G)$ is a nonabelian simple group. In the late 1990s, Huppert \cite{Hupp-I} posed a conjecture which asserts that the nonabelian simple groups are essentially characterized by the set of their character degrees.

\begin{conjecture}{(Huppert)} Let $G$ be a finite group and $H$ a finite nonabelian simple group such that the sets of character degrees of $G$ and $H$ are the same. Then $G \cong H \times A$, where $A$ is an abelian group.
\end{conjecture}

This conjecture is verified for Alternating groups, many of the simple groups of Lie type \cite{Tong12:L4,Wakefield09:L3,Wakefield11:2G2} and for all sporadic simple groups \cite{ADTW-Fi23,ADTW-2013,Hupp-II-VIII,HT}.

In this paper, we initiate an investigation on an extension of Huppert's conjecture for almost simple groups whose socle is the sporadic simple groups. A group $H$ is called \emph{almost simple} if there exists a nonabelian simple group $H_{0}$ such that $H_{0}\leq H\leq \Aut(H_{0})$. Indeed, this paper is devoted to studying finite groups with the same character degrees as almost simple groups $H$ with socle $H_{0}$ being one of the Mathieu groups:

\begin{theorem}\label{thm:main}
Let $G$ be a finite group, and let $H$ be an almost simple group whose socle is one of the Mathieu groups. If $\cd(G)=\cd(H)$, then there exists an abelian group $A$ such that $G/A$ is isomorphic to $H$.
\end{theorem}

In order to prove Theorem~\ref{thm:main}, we establish the following steps which Huppert introduced in \cite{Hupp-I}. Let $H$ be an almost simple group with socle $H_{0}$, and let $G$ be a group with the same character degrees as $H$. Then we show that
\begin{description}
  \item[Step 1.] $G'=G''$;
  \item[Step 2.] if $G'/M$ is a chief factor of $G$, then $G'/M$ is isomorphic to $H_{0}$;
  \item[Step 3.] if $\theta \in \Irr(M)$ with $\theta(1)=1$, then $I_{G'}(\theta)=G'$ and so $M=M'$;
  \item[Step 4.] $M=1$ and $G'\cong H_{0}$;
  \item[Step 5.] $G/C_{G}(G')$ is isomorphic to $H$.
\end{description}

Note that to prove Step 2, we determine all  finite simple groups whose irreducible character degrees divide some irreducible character degrees of almost simple groups with socle sporadic simple groups, and by Proposition~\ref{prop:simple}, all such simple groups are listed in Table~\ref{tbl:simple}. This result somehow relates to \cite[Theorem 1]{Tong12:CAlg-Alt-Spor} for sporadic simple groups.




\section{Preliminaries}\label{sec:prem}

In this section, we present some useful results to prove Theorem~\ref{thm:main}. We first establish some definitions and notation.

Throughout this paper all groups are finite. Recall that a  group $H$ is said to be an almost simple group with socle $H_{0}$ if $H_{0}\leq H\leq \Aut(H_{0})$, where $H_{0}$ is a nonabelian simple group. For a positive integer $n$, $\pi(n)$ denotes the set of all prime divisors of $n$. If $G$ is a group, we will write $\pi(G)$ instead of $\pi(|G|)$. If $N\unlhd G$ and $\theta\in \textrm{Irr}(N)$, then the inertia group of
$\theta$ in $G$ is denoted by $I_G(\theta)$ and is defined by $I_G(\theta)=\{g\in G\ |\ \theta^g=\theta\}$. If the character $\chi=\sum_{i=1}^k e_i\chi_i$, where each $\chi_i$ is an irreducible character of $G$ and $e_i$ is a nonnegative integer, then those $\chi_i$ with $e_i>0$ are called the \textsl{irreducible constituents} of $\chi$. The set of all irreducible constituents of $\theta^G$ is denoted by $\textrm{\Irr}(G|\theta)$.  All further notation and definitions are standard and could be found in \cite{HuppBook,Isaacs-book}.

\begin{lemma}~\cite[Theorems 19.5 and 21.3]{HuppBook}\label{lem:gal}
Suppose $N\unlhd G$ and $\chi\in {\rm{\Irr}}(G)$.
\begin{enumerate}
  \item[(a)] If $\chi_N=\theta_1+\theta_2+\cdots+\theta_k$ with $\theta_i\in {\rm{\Irr}}(N)$, then $k$ divides $|G/N|$. In particular, if $\chi(1)$ is prime to $|G/N|$, then $\chi_N\in {\rm{\Irr}}(N)$.
  \item[(b)] (Gallagher's Theorem) If $\chi_N\in {\rm{\Irr}}(N)$, then $\chi\psi\in {\rm{\Irr}}(G)$ for all $\psi\in {\rm{\Irr}}(G/N)$.
\end{enumerate}
\end{lemma}

\begin{lemma}~\cite[Theorems 19.6 and 21.2]{HuppBook}\label{lem:clif}
Suppose $N\unlhd G$ and $\theta\in {\rm{\Irr}}(N)$. Let $I=I_G(\theta)$.
\begin{enumerate}
  \item[(a)]  If $\theta^I=\sum_{i=1}^k\phi_i$ with $\phi_i\in {\rm{\Irr}}(I)$, then $\phi_i^G\in {\rm{\Irr}}(G)$. In particular, $\phi_i(1)|G:I|\in \cd(G)$.
  \item[(b)] If $\theta$ extends to $\psi\in {\rm{\Irr}}(I)$, then $(\psi\tau )^G\in {\rm{\Irr}}(G)$ for all $\tau\in {\rm{\Irr}}(I/N)$. In particular, $\theta(1)\tau(1)|G:I|\in {\rm{\cd}}(G)$.
  \item[(c)] If $\rho \in {\rm{\Irr}}(I)$ such that $\rho_N=e\theta$, then $\rho=\theta_0\tau_0$, where $\theta_0$ is a character of an irreducible projective representation of $I$ of degree $\theta(1)$ and $\tau_0$ is a character of an irreducible projective representation of $I/N$ of degree $e$.
\end{enumerate}
\end{lemma}

\begin{lemma}~\cite[Lemma~3]{HT}\label{lem:factsolv} Let $G/N$ be a solvable factor group of
$G$ minimal with respect to being nonabelian. Then two cases can occur.
\begin{enumerate}
  \item[(a)] $G/N$ is an $r$-group for some prime $r$. Hence there exists $\psi\in {\rm{\Irr}}(G/N)$ such that $\psi(1)=r^b>1$. If $\chi\in {\rm{\Irr}}(G)$ and $r\nmid \chi(1)$, then $\chi\tau\in {\rm{\Irr}}(G)$ for all $\tau\in {\rm{\Irr}}(G/N)$.
  \item[(b)]  $G/N$ is a Frobenius group with an elementary abelian Frobenius kernel $F/N$. Then $f=|G:F|\in {\rm{\cd}}(G)$ and $|F/N|=r^a$ for some prime $r$, and $a$ is the smallest integer such that $r^a\equiv 1 \mod f$. If $\psi\in {\rm{\Irr}}(F)$, then either $f\psi(1)\in {\rm{\cd}}(G)$ or $r^a$ divides $\psi(1)^2$. In the latter case, $r$ divides $\psi(1)$.
  \begin{enumerate}
    \item[(1)] If no proper multiple of $f$ is in ${\rm{\cd}}(G)$, then $\chi(1)$ divides $f$ for all $\chi\in {\rm{\Irr}}(G)$ such that $r\nmid \chi(1)$, and if  $\chi\in {\rm{\Irr}}(G)$ such that $\chi(1)\nmid f$, then $r^a\mid \chi(1)^2$.
     \item[(2)] If $\chi\in {\rm{\Irr}}(G)$ such that no proper multiple of $\chi(1)$ is in ${\rm{\cd}}(G)$, then either $f$ divides $\chi(1)$ or $r^a$ divides $\chi(1)^2$. Moreover if $\chi(1)$ is divisible by no nontrivial proper character degree in $G$, then $f=\chi(1)$ or $r^a\mid \chi(1)^2$.
  \end{enumerate}
\end{enumerate}
\end{lemma}


\begin{lemma}\label{lem:exten}
Let $G$ be a finite group.
\begin{enumerate}
  \item[(a)] If $G$ is a nonabelian simple group, then there exists a nontrivial irreducible character $\varphi$ of $G$ that extends to ${\rm{\Aut}}(G)$.
  \item[(b)] If $N$ is a minimal normal subgroup of $G$ so that $N\cong S^k$, where $S$ is a nonabelian simple group, and  $\varphi\in {\rm{\Irr}}(S)$ extends to ${\rm{\Aut}}(S)$, then $\varphi^k\in {\rm{\Irr}}(N)$ extends to $G$.
\end{enumerate}
\end{lemma}
\begin{proof}
Part (a) follows from~\cite[Theorems~2-4]{Bia} and to prove part (b) see~\cite[Lemma 5]{Bia}.
\end{proof}

\begin{lemma}~\cite[Lemma 6]{Hupp-I}\label{lem:schur}
Suppose that $M\unlhd G'=G''$ and for every $\lambda\in {\rm{\Irr}}(M)$ with $\lambda(1)=1$, $\lambda^g=\lambda$ for all $g\in G'$. Then $M'=[M,G']$ and $|M/M'|$ divides the order of the Schur multiplier of $G'/M$.
\end{lemma}

\begin{lemma}~\cite[Theorem D]{Moreto}\label{lem:sol}
Let $N$ be a normal subgroup of a finite group $G$ and let $\varphi \in \Irr(N)$ be
$G$-invariant. Assume that $\chi(1)/\varphi(1)$ is odd, for all $\chi(1)\in \Irr(G|\varphi)$. Then $G/N$ is solvable.
\end{lemma}

\section{Degree properties of almost simple groups with socle sporadic }\label{sec:simple}

In this section, we determine all finite simple groups whose irreducible character degrees divide some irreducible character degrees of almost simple groups whose  socles are sporadic simple groups.

\begin{proposition}\label{prop:simple}
Let $S$ be a simple group, and let $H$ be an almost simple group whose socle is a sporadic simple group. Then the character degrees of $S$ divides some character degrees of $H$ if and only if $H$ and $S$ are as in Table~\ref{tbl:simple}.
\end{proposition}
\begin{proof}
Suppose that $H$ is an almost simple group with socle a sporadic simple group $H_{0}$. Suppose also that $S$ is a simple group whose degrees divide some degrees of $H$. Then the prime divisors of the degrees of $S$ are exactly those primes dividing $|S|$. Therefore, $\pi(S)\subseteq \pi(H)$, and hence by~\cite{Zavar-2008}, we know all such possible simple groups $S$. We only need to check if the degrees of $S$ divide some degrees of $H$ and this could mainly be done by \cite{Atlas,Gap4}. Since our arguments are similar for each group $H$, we only give a detailed proof for the cases where $H$ is $M_{12}:2$ or $M_{22}:2$ which will be frequently used and referred to in Sections~\ref{sec:M12.2} and~\ref{sec:M22.2}.

Suppose first $H=M_{12}:2$. Then $\pi(S)\subseteq \{2,3,5,11\}$, and so by~\cite{Zavar-2008}, $S$ is isomorphic to one of the simple groups $A_5$, $A_6$, $L_2(11)$, $U_5(2)$, $S_4(3)$, $M_{11}$ and $M_{12}$. Note that $U_5(2)$ and $S_4(3)$ have degrees $220$ and $81$, respectively. Therefore, $S$ is isomorphic to $A_5$, $A_6$, $L_2(11)$, $M_{11}$ or $M_{12}$.

Suppose now $H=M_{22}:2$. Then $\pi(S)\subseteq \{2,3,5,7,11\}$, and so by~\cite{Zavar-2008}, $S$ is isomorphic to one of the simple groups in $\mathcal{A}_{1}\cup \mathcal{A}_{2}\cup\mathcal{A}_{3}$, where
\begin{align*}
  \mathcal{A}_{1}:=\{&A_5 , A_6,A_7, L_2(7) , L_2(8), M_{22}\}, \\
  \mathcal{A}_{2}:=\{&A_{8}, L_{3}(4), L_{2}(49), U_3(3), S_4(7), M_{11}\},\\
  \mathcal{A}_{3}:=\{&A_9, A_{10}, A_{11}, A_{12}, L_2(11), U_3(5),U_4(3),U_5(2), U_6(2),\\
                     &S_4(3), S_6(2), O_8^{+}(2), M_{12}, McL, HS, J_2\}.
\end{align*}
If $S\in \mathcal{A}_{2}$, then $S$ has a degree divisible by $25$, $27$, $44$, $49$ or $64$, which is a contradiction. If $S\in \mathcal{A}_{3}$, then $S$ has a degree which is divisible by $12$, which is also a contradiction. Therefore $S\in \mathcal{A}_{1}$ as claimed.
\end{proof}

\begin{longtable}{lp{9.5cm}}
  \caption{Simple groups $S$ whose irreducible character degrees divide some character degrees of almost simple groups $H$ with socle sporadic simple groups.}\label{tbl:simple}\\
    \hline
    \multicolumn{1}{c}{$H$} & \multicolumn{1}{c}{$S$} \\
    \hline
    \endfirsthead
    \multicolumn{2}{c}%
    {\tablename\ \thetable\ -- Continued} \\
    \hline
    \multicolumn{1}{c}{$H$} & \multicolumn{1}{c}{$S$} \\
    \hline
    \endhead
    \hline \multicolumn{2}{r}{Continued}\\
    \endfoot
    \hline
    \endlastfoot
    $M_{11}$ &
    $A_n$ for $n=5,6$, $M_{11}$   \\
    $M_{12}$, $M_{12}:2$ &
    $A_n$ for $n=5,6$,    $L_2(11)$,  $M_{11}$, $M_{12}$ \\
    $M_{22}$, $M_{22}:2$ &
    $A_n$ for $n=5,6,7$, $L_2(q)$ for $q=7,8$, $M_{22}$\\
    $M_{23}$ &
    $A_n$ for $n=5,6,7$, $L_2(q)$ for $q=7,8$, $M_{11}$, $M_{23}$ \\
    $M_{24}$ &
    $A_n$ for $n=5,6,7,8$, $L_2(q)$ for $q=7,8,11,23$, $L_3(4)$, $U_3(3)$, $M_{11}$, $M_{24}$ \\
    $J_1$ &
    $A_5$, $L_2(q)$ for $q=7,11$, $J_1$ \\
    $J_2$ &
    $A_n$ for $n=5,6,7$, $L_2(q)$ for $q=7,8$, $U_3(3)$, $J_2$ \\
    $J_2:2$ & $A_8$, $L_{3}(4)$, and all $S$ in $J_{2}$, \\
    $J_3$, $J_3:2$ &
    $A_n$ for $n=5,6$, $S_4(3)$, $L_2(q)$ for $q=16,17,19$, $J_3$ \\
    $J_4$ &
    $A_n$ for $n=5,6,7,8$, $L_2(q)$ for $q=7,8,11,23,29,31,32,43$, $L_3(4)$, $L_5(2)$, $U_3(3)$, $U_3(11)$, $M_{11}$, $M_{12}$, $M_{22}$, $J_4$ \\
    $HS$, $HS:2$&
    $A_n$ for $n=5,6,7,8$, $L_2(q)$ for $q=7,8,11$, $L_3(4)$, $M_{11}$, $M_{22}$, $HS$ \\
    $McL$, $McL:2$&
    $A_n$ for $n=5,6,7,8$, $L_2(q)$ for $q=7$, $8$, $11$, $L_3(4)$, $U_3(3)$, $U_3(5)$ $S_4(3)$, $M_{11}$, $McL$ \\
    $Suz$, &
    $A_n$ for $n=5,\ldots,10$, $L_2(q)$ for $q=7$, $8$, $11$, $13$, $25$, $27$, $64$,
    $L_3(3)$, $L_3(4)$, $L_3(9)$, $L_4(3)$,
    $U_3(3)$, $U_3(4)$, $U_4(2)$, $U_4(3)$,
    $U_5(2)$,
    $S_6(2)$,
    $^2B_2(8)$, $G_2(3)$,$^2F_4(2)'$,
    $M_{11}$, $M_{12}$, $M_{22}$, $Suz$, $J_2$ \\
    $Suz:2$ &
    $A_{11}$, and all $S$ in $Suz$, \\
    $Co_3$ &
    $A_n$ for $n=5$, $6$, $7$, $8$, $9$, $11$, $L_2(q)$ for $q=7$, $8$, $11$, $23$, $L_3(4)$, $U_3(q)$ for $q=3,5$ , $U_4(3)$, $S_4(3)$, $S_6(2)$, $M_{11}$, $M_{12}$, $M_{22}$, $M_{23}$, $M_{24}$, $Co_3$ \\
    $Co_2$ &
    $A_n$ for $n=5,\ldots,11$, $L_2(q)$ for $q=7,8,11,23$, $L_3(4)$, $U_3(q)$ for $q=3,5$, $U_4(3)$, $U_5(2)$, $S_4(3)$, $S_6(2)$, $M_{11}$, $M_{12}$, $M_{22}$, $M_{23}$, $M_{24}$ $J_2$, $Co_2$ \\
    $Co_1$ &
    $A_n$ for $n=5,\ldots,16$, $L_2(q)$ for $q=7$, $8$, $11$, $13$, $23$, $25$, $27$, $49$, $64$, $L_3(q)$ for $q=3$, $4$, $9$, $L_4(3)$, $U_3(q)$ for $q=3$, $4$, $5$, $U_4(3)$,$U_5(2)$, $U_6(2)$, $S_4(q)$ for $q=3$, $5$, $8$, $S_6(q)$ for $q=2,3$, $^2B_2(8)$, $O_7(3)$, $O^+_8(2)$, $G_2(q)$ for $q=3,4$, $^3D_4(2)$, $^2F_4(2)'$, $M_{11}$, $M_{12}$, $M_{22}$, $M_{23}$, $M_{24}$, $McL$,
    $J_2$, $HS$, $Co_1$, $Co_{3}$ \\
    $He$, $He:2$ &
    $A_n$ for $n=5,6,7,8$,  $L_2(q)$ for $q=7,8,16,17,49$, $ L_3(4)$, $U_3(3)$, $S_4(4)$, $He$ \\
    $Fi_{22}$ &
    $A_n$ for $n=5,\ldots,11$, $A_{13}$, $L_2(q)$ for $q=7$, $8$, $11$, $13$, $25$, $27$, $64$, $L_3(q)$ for $q=3,4,9$, $L_4(3)$, $U_3(q)$ for $q=3$, $4$, $U_{4}(3)$, $U_5(2)$, $U_6(2)$, $S_4(3)$, $S_6(2)$, $^2B_2(8)$, $O^+_8(2)$, $G_2(q)$ for $q=3,4$, $M_{11}$, $M_{12}$, $M_{22}$, $J_2$, $McL$, $Fi_{22}$\\
    $Fi_{22}:2$ & $S_6(3)$, $O_7(3)$, and all $S$ in $Fi_{22}$.\\
    $Fi_{23}$ &
    $A_n$ for $n=5,\ldots,13$, $L_2(q)$ for $q=7$, $8$, $11$, $13$, $16$, $17$, $23$, $25$, $27$, $64$, $L_3(q)$ for $q=3$, $4$, $9$, $16^{2}$, $L_4(q)$ for $q=3$, $4$,  $U_3(q)$ for $q=3$, $4$, $U_4(q)$ for $q=3$, $4$,  $U_5(2)$, $S_4(3)$, $S_4(4)$, $S_6(2)$, $S_6(3)$, $^2B_2(8)$, $O_7(3)$, $O^+_8(2)$, $O^-_8(2)$, $G_2(q)$ for $q=3,4$, $^2F_4(2)'$,
    $M_{11}$, $M_{12}$, $M_{22}$, $M_{23}$, $M_{24}$, $HS$, $J_2$, $Fi_{23}$\\
    $Fi'_{24}$, $Fi'_{24}:2$  &
    $A_n$ for $n=5,\ldots,14$, $L_2(q)$ for $q=7$, $8$, $11$, $13$, $16$, $17$, $23$, $25$, $27$, $29$, $49$, $64$, $L_3(q)$ for $q=3$, $4$, $9$, $16^{2}$, $L_4(q)$ for $q=3$, $4$, $L_{6}(3)$, $U_3(q)$ for $q=3$, $4$, $U_{4}(3)$, $U_5(2)$, $S_4(q)$ for $q=3$, $4$, $8$, $S_6(2)$, $S_6(3)$, $S_8(2)$, $^2B_2(8)$, $O_7(3)$, $O^\pm_8(2)$, $G_2(q)$ for $q=3$, $4$, $^3D_4(2)$, $^2F_4(2)'$, $M_{11}$, $M_{12}$, $M_{22}$, $M_{23}$, $M_{24}$, $He$, $J_2$, $Fi'_{24}$\\
    $Th$ &
    $A_n$ for $n=5,\ldots,10$, $L_2(q)$ for $q=7$, $8$, $13$, $19$, $25$, $27$, $31$, $49$, $64$, $L_3(q)$ for $q=3$, $4$, $5$, $9$, $L_4(3)$, $L_5(2)$, $U_3(q)$ for $q=3$, $4$, $5$, $8$, $S_4(3)$, $S_4(8)$,$S_6(2)$, $^2B_2(8)$, $^{3}D_{4}(2)$, $G_2(q)$ for $q=3$, $4$, $^2F_4(2)'$, $J_2$, $Th$ \\
    $Ru$ &
    $A_n$ for $n=5,\ldots,8$, $L_2(q)$ for $q=7$, $8$, $13$, $25$, $27$, $29$, $64$, $L_3(q)$ for $q=3$, $4$, $U_3(q)$ for $q=3$, $4$, $5$, $J_2$, $Ru$ \\
    $Ly$  &
    $A_n$ for $n=5,\ldots,9$, $A_{11}$, $L_2(q)$ for $q=7$, $8$, $11$, $31$, $32$, $L_3(q)$ for $q=4$, $5$, $U_3(q)$ for $q=3$, $5$, $S_4(3)$, $M_{11}$, $M_{12}$, $M_{22}$, $J_2$, $McL$, $HS$, $Ly$\\
    $HN$,  $HN:2$  &
    $A_n$ for $n=5,\ldots,10$, $L_2(q)$ for $q=7$, $8$, $11$, $19$, $L_3(q)$ for $q=4$. $19$, $L_{4}(7)$, $U_3(q)$ for $q=3$, $5$, $8$, $U_4(3)$, $S_4(3)$,$O^{+}_{8}(2)$, $M_{11}$, $M_{12}$, $M_{22}$, $J_1$, $J_{2}$, $HS$, $HN$ \\
    $O'N$, $O'N:2$ &
    $A_n$ for $n=5,\ldots,8$, $L_2(q)$ for $q=7$, $8$, $11$, $16$, $17$, $19$, $31$, $32$, $L_3(q)$ for $q=4$, $7$, $U_3(q)$ for $q=3$, $8$, $S_{4}(3)$, $S_6(2)$, $M_{11}$, $M_{12}$, $M_{22}$, $O'N$ \\
    $B$ &
    $A_n$ for $n=5,\ldots,28$, $L_2(q)$ for $q=7$, $8$, $11$, $13$, $16$, $17$, $19$, $23$, $27$, $31$, $32$, $47$, $49$, $64$, $125$, $L_3(q)$ for $q=3$, $4$, $5$, $9$, $16$, $25$, $L_4(q)$ for $q=3$, $4$, $5$, $L_5(2)$, $L_{5}(4)$, $L_6(2)$, $L_{6}(4)$, $U_3(q)$ for $q=3,4,5,8$, $U_4(q)$ for $q=2$, $3$, $4$, $5$, $8$, $U_5(2)$, $U_6(2)$, $S_4(q)$ for $q=4,5,7,8$, $G_2(q)$ for $q=3$, $4$, $5$, $O_7(3)$, $O^\pm_8(2)$, $O^+_8(3)$, $O^\pm_{10}(2)$, $O^\pm_{12}(2)$, $^2F_4(2)'$, $^3D_4(2)$, $M_{11}$, $M_{12}$, $M_{22}$, $M_{23}$, $M_{24}$, $J_1$, $J_2$ , $J_3$, $HS$, $McL$, $Suz$, $Fi_{22}$, $Co_3$, $Co_2$, $Th$, $B$ \\
    $M$ &
    $A_n$ for $n=5,\ldots,36$, $L_2(q)$ for $q=7$, $8$, $11$, $13$, $16$, $17$, $19$, $23$, $25$, $27$, $29$, $31$, $32$, $41$, $47$, $49$, $59$, $64$, $71$, $81$, $169$, $1024$, $L_3(q)$ for $q=3$, $4$, $5$, $7$, $9$, $16$, $19$, $25$, $L_4(q)$ for $q=3$, $4$, $5$, $7$, $9$, $L_5(q)$ for $q=2$, $4$, $L_6(q)$, for  $q=2$, $3$, $4$, $U_3(q)$ for $q=3$, $4$, $5$, $8$, $27$, $U_4(q)$ for $q=2$, $3$, $4$, $5$, $8$, $9$, $U_5(q)$ for $q=2$, $4$, $U_6(2)$ for $q=2$, $4$, $S_4(q)$ for $q=4$, $5$, $7$, $8$, $9$, $S_6(q)$ for $q=2$, $3$, $4$, $5$, $S_8(2)$, $S_{10}(2)$, $^2B_2(8)$, $^2B_2(32)$, $G_2(q)$ for $q=3,4,5$, $O_7(3)$, $O_7(5)$, $O_9(3)$,$O^\pm_8(2)$, $O^\pm_8(3)$, $O^\pm_{10}(2)$, $O^\pm_{10}(3)$, $O^\pm_{12}(2)$, $^{2}G_{2}(27)$, $^2F_4(2)'$, $F_4(2)$, $^3D_4(2)$, $M_{11}$, $M_{12}$, $M_{22}$, $M_{23}$, $M_{24}$, $J_1$, $J_2$, $J_3$, $HS$, $McL$, $Suz$, $Fi_{22}$, $Co_3$, $Co_2$, $Th$, $He$, $O'N$, $Ru$, $M$ \\
    \hline
\end{longtable}

\section{Groups with socle $M_{12}$}\label{sec:M12.2}

In this section, we prove Theorem~\ref{thm:main} for almost simple group $H$ whose socle is $H_{0}:=M_{12}$. By \cite{Hupp-II-VIII}, Theorem~\ref{thm:main} is proved when $H=H_{0}=M_{12}$. Therefore, we only need to deal with the case where $H:=M_{12}:2$. For convenience, we mention some properties of $H$ and $H_{0}$ some of which can be obtained in \cite[pp. 31-33]{Atlas}.

\begin{lemma}\label{lem:M12.2}
Let $H_{0}:= M_{12}$ and $H:=M_{12}:2$. Then
\begin{enumerate}
  \item[(a)] The Schur multiplier and the group of outer automorphisms of $H_{0}$ are $\Zbb_{2}$;
  \item[(b)] The degrees of irreducible characters of $H$ are
      \begin{align*}
      1& &
      45&=3^2\cdot5 &
      66&=2\cdot3\cdot11  &
      120&=2^3\cdot3\cdot5\\
      22&=2\cdot 11 &
      54&=2\cdot3^3  &
      99&=3^2\cdot11  &
      144&=2^4\cdot3^2\\
      32&=2^5 &
      55& =5\cdot11 &
      110&=2\cdot 5\cdot 11 &
      176&=2^4\cdot11
      \end{align*}
  \item[(c)] If $K$ is a maximal subgroup of $H_{0}$ whose index in $H_{0}$ divides some degrees $\chi(1)$ of $H$, then one of the following occurs:
      \begin{enumerate}
        \item[(i)] $K\cong M_{11}$ and $\chi(1)/|H_{0}:K|$ divides $2\cdot5$ or $2^2\cdot3$;
        \item[(ii)] $K\cong M_{10}:2$ and  $\chi(1)/|H_{0}:K|=1$;
        \item[(iii)] $K\cong L_2(11)$ and  $\chi(1)/|H_{0}:K|=1$.
      \end{enumerate}
  \item[(d)] If $S$ is a finite nonabelian simple group whose irreducible character degrees divide some degrees of $H$, then $S$ is isomorphic to $A_5$, $A_6$, $L_2(11)$, $M_{11}$ or $M_{12}$.
\end{enumerate}
\end{lemma}
\begin{proof}
Parts (a) and (b) follows from \cite[pp. 31-33]{Atlas} and part (d) follows from Proposition~\ref{prop:simple} and Table~\ref{tbl:simple}. Part (c) is a straightforward calculation.
\end{proof}

\subsection{Proof of Theorem~\ref{thm:main} for $M_{12}:2$}\label{sec:proof-M12.2}

As noted before, by \cite{Hupp-II-VIII}, we may assume that $H:=M_{12}:2$. We further assume that $G$ is a finite group with $\cd(G)=\cd(M_{12}:2)$. The proof of Theorem~\ref{thm:main} follows from the following lemmas.

\begin{lemma}\label{lem:m12-1}
$G'=G''$.
\end{lemma}
\begin{proof}
Assume the contrary. Then there is a normal subgroup $N$ of $G$, where $N$ is maximal such that $G/N$ is a nonabelian solvable group. Now we apply Lemma~\ref{lem:factsolv} and we have one of the following cases:\smallskip

\noindent (a) $G/N$ is a $r$-group with $r$ prime. In this case, $G/N$ has an irreducible character $\psi$ of degree $r^b>1$, and so does $G$. Since $M_{12}:2$ has an irreducible character of degree $32$, we conclude that $r=2$. Let now $\chi \in \Irr(G)$ with $\chi(1)=99$. Then Lemma~\ref{lem:gal}(a) implies that $\chi_N\in \Irr(N)$, and so by Lemma~\ref{lem:gal}(b), $G$ has an irreducible character of degree $99\psi(1)$, which is a contradiction.\smallskip

\noindent (b) $G/N$ is a Frobenius group with kernel $F/N$. Then $|G:F| \in \cd(G)$ divides $r^{a}-1$, where $|F/N|=r^{a}$. Let $|G:F|\in \{2\cdot 11,5\cdot 11\}$ and $\chi(1)=2^4\cdot 3^2$. Then Lemma~\ref{lem:clif}(b) implies that $r^a$ divides $\chi^2(1)=2^8\cdot3^4$, and it is impossible as $|G:F|$ does not divide $r^a-1$, for every divisor  $r^a$ of $2^8\cdot3^4$. Let now $|G:F|\not \in \{2\cdot11,5\cdot11\}$. Then no proper multiple of $|G:F|$ is in $\cd(G)$. If $r=2$, then by Lemma~\ref{lem:clif}(a), both $3^2\cdot11$ and $3^2\cdot5$ must divide $|G:F|$, which is a contradiction. If $r=3$, then by Lemma~\ref{lem:clif}(a), $|G:F|$ is divisible by $2^5$ and $2^4\cdot11$, which is a contradiction. Similarly, if $r\neq 2$, $3$, then $2^5$ and $2^4\cdot3^2$ divide $|G:F|$, which is a contradiction.
\end{proof}

\begin{lemma}\label{lem:M12-2}
Let  $G'/M$ be a chief factor of $G$. Then $G'/M \cong M_{12}$.
\end{lemma}
\begin{proof}
Suppose $G'/M\cong S^k$, where $S$ is a nonabelian simple group for some positive integer $k$. Since $S$ is a finite nonabelian simple group whose irreducible character degrees divide some degrees of $M_{12}:2$, by Lemma~\ref{lem:M12.2}(d), $S$ is isomorphic to one of the groups $A_5$, $A_6$, $M_{11}$, $M_{12}$ or $L_2(11)$. If $S$ is isomorphic to one of the groups $A_5$, $A_6$, $M_{11}$, $M_{12}$, $L_2(11)$, then $k=1$ as $G$ has no degree divisible by $5^2$. Assume that $S$ is isomorphic to $A_5$ or $A_6$, then $G'/M$ has a character $\psi$ of degree $5$. If $\chi$ is an irreducible constituent of $\psi^{G/M}$, then $\chi(1)=t\psi(1)=5t$, where $t$ divides $|\Out(S)|$. Consequently, $G$ has a character of degree at most $20$, which is a contradiction. Similarly, in the case where $S$ is isomorphic to $M_{11}$ or $L_2(11)$, the factor group $G/M$ has a character of degree $10t$ with $t=1,2$, and this implies that $G$ has a character of degree at most $20$, which is a contradiction.
Therefore $G'/M$ is isomorphic to $M_{12}$.
\end{proof}
\begin{lemma}\label{lem:M12-3}
Let $\theta \in \Irr(M)$ with $\theta(1)=1$. Then $I_{G'}(\theta)=G'$  and $M=M'$.
\end{lemma}
\begin{proof}
Suppose  $I=I_{G'}(\theta)<G'$. By Lemma~\ref{lem:clif}, we have $\theta^I= \sum_{i=1}^{k} \phi_i$ where $\phi_i \in \Irr(I)$ for $ i=1,2,...,k$. Let $U/M$ be a maximal subgroup of $G'/M\cong M_{12}$ containing $I/M$ and set $t:= |U:I|$. It follows from Lemma~\ref{lem:clif}(a) that $\phi_i(1)|G':I| \in \cd(G')$, and so $t\phi_i(1)|G':U|$ divides some degrees of $G$. Then $|G':U|$ must divide some character degrees of $G$, and hence by Lemma~\ref{lem:M12.2}(c) one of the following holds. \smallskip

\noindent (i) Suppose  $U/M\cong M_{11}$. Then $t\phi_i(1)$ divide $2\cdot5$ or $2^2\cdot3$. If $t=1$, then $I/M\cong M_{11}$. Since $M_{11}$ has trivial Schur multiplier, it follows that $\theta$ extends to $\theta _0\in \Irr(I)$, and so by Lemma~\ref{lem:clif}(b) $(\theta_0\tau )^{G'}\in \Irr(G')$, for all $\tau \in \Irr(I/M)$. For $\tau (1)=55\in \cd(M_{11})$, it turns out that $12\cdot55 \cdot\theta_0(1)$ divide some degrees of $G$, which is a contradiction. Therefore, $t\neq 1$, and hence the index of a maximal subgroup of $U/M \cong M_{11}$ containing $I/M$ must divide $2\cdot5$ or $2^2\cdot3$. This implies that $t\phi_i(1)$ divides $2^2\cdot3$ and $I/M \cong L_2(11)$. In particular, $\phi_{i}(1)=1$. Thus $\theta$ extends to a $\phi_{i}$, and so by Lemma~\ref{lem:clif}(b), $144\tau(1) \in \cd(G')$, for all $\tau \in \Irr(I/M)$. This leads us to a contradiction by taking $\tau(1)=10\in \cd(L_2(11))$.\smallskip

\noindent (ii) Suppose  $U/M\cong M_{10}:2$. In this case $t=1$, or equivalently, $I/M=U/M\cong M_{10}:2$. Moreover, $\phi_i(1)=1$, for all $i$. Then $\theta$ extends to $\phi_i \in \Irr(I)$, and so by Lemma~\ref{lem:clif}(b), $66 \tau (1)$ divides some degrees of $G$, for $\tau (1)=10$, which is a contradiction.\smallskip

\noindent (iii) Suppose  $U/M\cong L_2(11)$ and  $t=\phi_{i}(1)=1$, for all $i$. Then $I/M \cong L_2(11)$, and so $\theta$ extends to $\phi_i \in \Irr(I)$. Thus $144 \tau (1)\in \Irr(G')$, for all $\tau \in \Irr(I/M)$. This is impossible by taking $\tau (1)=10$.

Therefore, $I_{G'}(\theta)=G'$. By Lemma~\ref{lem:schur}, we have that $|M/M'|$ divides the order of Schur multiplier of $G'/M\cong M_{12}$ which is $2$. If $|M/M'|=2$, then $G'/M'$ is isomorphic to $2\cdot M_{12}$ which has a character of degree $32$ \cite[p. 33]{Atlas}. Therefore $M_{12}$ must have a degree divisible by $32$, which is a contradiction. Hence $|M/M'|=1$, or equivalently, $M=M'$.
\end{proof}
\begin{lemma}\label{lem:M12-4}
The subgroup $M$ is trivial, and hence $G'\cong M_{12}$.
\end{lemma}
\begin{proof}
By Lemmas~\ref{lem:M12-2} and \ref{lem:M12-3}, we have that $G'/M \cong M_{12}$ and $M=M'$. Suppose that $M$ is nonabelian, and let $N\leq M$ be a normal subgroup of $G'$ such that $M/N$ is a chief factor of $G'$. Then $M/N\cong S^{k}$, for some nonabelian simple group $S$. It follows from Lemma~\ref{lem:exten} that $S$ possesses a nontrivial irreducible character $\varphi$ such that $\varphi^{k}\in \Irr(M/N)$ extends to $G'/N$. By Lemma~\ref{lem:gal}(b), we must have $\varphi(1)^{k}\tau(1)\in \cd(G'/N)\subseteq \cd(G')$, for all $\tau \in \Irr(G'/M)$. Now we can choose $\tau\in G'/M$ such that $\tau(1)$ is the largest degree of $M_{12}$, and since $\varphi$ is nontrivial,  $\varphi(1)^{k}\tau(1)$ divides no degree of $G$, which is a contradiction. Therefore, $M$ is abelian, and since $M=M'$, we conclude that $M=1$.
\end{proof}

\begin{lemma}\label{lem:M12-5}
There exists an abelian group $A$ such that $G/A\cong M_{12}:2$.
\end{lemma}
\begin{proof}
Set $A:= C_G(G')$. Since $G'\cap A=1$ and $G'A\cong G' \times A$, it follows that $G'\cong G'A/A\unlhd G/A\leq \Aut(G')$. By Lemma~\ref{lem:M12-4}, we have $G'\cong M_{12}$, and so we conclude that $G/A$ is isomorphic to $M_{12}$ or $M_{12}:2$. In the case where $G/A$ is isomorphic to $M_{12}$, we must have $G\cong A\times M_{12}$. This is impossible as $32 \in \cd(G)$ but $M_{12}$ has no character of degree $32$. Therefore, $G/A$ is isomorphic to $M_{12}:2$.
\end{proof}

\section{Groups with socle $M_{22}$}\label{sec:M22.2}

In this section, we prove Theorem~\ref{thm:main} for almost simple group $H$ whose socle is $H_{0}:=M_{22}$. Note that Theorem~\ref{thm:main} is proved for $H=H_{0}=M_{22}$, see \cite{Hupp-II-VIII}. Therefore, we only need to focus on case where $H:=M_{22}:2$. For convenience, we mention some properties of $H$ and $H_{0}$ some of which can be obtained from \cite[pp. 39-41]{Atlas}.

\begin{lemma}\label{lem:M22.2}
Let $H_{0}:=M_{22}$ and $H:=M_{22}:2$. Then
\begin{enumerate}
  \item[(a)] The Schur multiplier $H_{0}$ is $\Zbb_{12}$ and the group of outer automorphisms of $H_{0}$ is $\Zbb_{2}$;
  \item[(b)] The degrees of irreducible characters of $H$ are
      \begin{align*}
      1& &
      55&=5\cdot 11 &
      210&=2\cdot 3\cdot 5\cdot 7&
      385&=5\cdot 7\cdot 11\\
      21&=3\cdot 7 &
      99&=3^{2}\cdot 11  &
      231&=3\cdot7\cdot11 &
      560&=2^{4}\cdot 5\cdot 7\\
      45&=3^{2}\cdot 5 &
      154&=2\cdot 7\cdot 11  &
      &&
      \end{align*}
  \item[(c)] If $K$ is a maximal subgroup of $H_{0}$ whose index in $H_{0}$ divides some degrees $\chi(1)$ of $H$, then one of the following occurs:
      \begin{enumerate}[{ \quad (i)}]
        \item $K\cong L_3(4)$ and $\chi(1)/|H_{0}:K|$ divides $7$;
        \item $K\cong 2^4: S_5$ and  $\chi(1)/|H_{0}:K|=1$;
        \item $K\cong 2^4:A_6$ and  $\chi(1)/|H_{0}:K|$ divides $2$, $3$ or $5$.
      \end{enumerate}
  \item[(d)] If $S$ is a finite nonabelian simple group whose irreducible character degrees divide some degrees of $H$, then $S$ is isomorphic to $A_5$, $A_{6}$, $A_{7}$, $L_2(7)$, $L_2(8)$ or $M_{22}$.
\end{enumerate}
\end{lemma}
\begin{proof}
Parts (a)-(b) follows from \cite[pp. 31-33]{Atlas} part (d) follows from Proposition~\ref{prop:simple} and Table~\ref{tbl:simple}. Part (c) is a straightforward calculation.
\end{proof}

\subsection{Proof of Theorem~\ref{thm:main} for $M_{22}:2$}\label{sec:proof-M12.2}

Theorem~\ref{thm:main} is true for the Mathieu group $M_{22}$ by \cite{Hupp-II-VIII}. It remains to assume that $H:=M_{22}:2$. In what follows assume that $G$ is a finite group with $\cd(G)=\cd(M_{22}:2)$. The proof follows from the following lemmas.

\begin{lemma}\label{lem:m22-1}
$G'=G''$.
\end{lemma}
\begin{proof}
Assume the contrary. Then there is a normal subgroup $N$ of $G$ where $N$ is a maximal such that $G/N$ is a nonabelian solvable group. We now apply Lemma $\ref{lem:factsolv}$, and so we have the following two cases: \smallskip

\noindent (a) Suppose that $G/N$ is a $r$-group with $r$ prime. Then it has an irreducible character $\psi$ of degree $r^b>1$. It is impossible as the group $M_{22}:2$ does not have a irreducible character of prime power degree.\smallskip

\noindent (b) $G/N$ is a Frobenius group with kernel $F/N$. Then $|G:F| \in \cd(G)$ divides $r^{a}-1$, where $|F/N|=r^{a}$. Let $|G:F|\in\{3\cdot7,5\cdot11\}$ and $\chi(1)=2\cdot7\cdot11$. Then Lemma~\ref{lem:clif}(b) implies that $r^a$ divides $\chi^2(1)=2^2\cdot7^2\cdot11^2$, and it is impossible as $|G:F|$ does not divide $r^a-1$, for every divisor $r^a$ of $2^2\cdot7^2\cdot11^2$. Let now $|G:F|\not \in \{3\cdot7,5\cdot11\}$. Then no proper multiple of $|G:F|$ is in $\cd(G)$.
If $r=3$, then by Lemma ~\ref{lem:clif}(a), both $5\cdot7\cdot11$ and $2^4\cdot5\cdot7$ must divide $|G:F|$, which is a contradiction. If $r=5$, then by Lemma ~\ref{lem:clif}(a), both $3\cdot7\cdot11$ and $2\cdot7\cdot11$ must divide $|G:F|$, which is a contradiction. If $r=11$, then by Lemma~\ref{lem:clif}(a), $|G:F|$ is divisible by $2\cdot3\cdot5\cdot7$ and $2^4\cdot5\cdot7$, which is a contradiction. Similarly, if $r\neq3$, $5$, and $11$, then $3^2\cdot11$ and $5\cdot11$ divide $|G:F|$, which is impossible. Therefore, $G'=G''$.
\end{proof}

\begin{lemma}\label{lem:M22-2}
Let $G'/M$ be a chief factor of $G$. Then $G'/M \cong M_{22}$.
\end{lemma}
\begin{proof}
Suppose $G'/M \cong S^k$, where $S$ is a nonabelian simple group and for some positive integer $k$. Since $S$ is a finite nonabelian simple group whose irreducible character degrees divide some degrees of $M_{22}:2$, by Lemma~\ref{lem:M22.2}(d), the group $S$ is isomorphic to $A_5$, $A_6$, $A_7$, $L_2(7)$, $L_2(8)$ or $M_{22}$. Observe that $G$ has no degree divisible by $25$ and $49$. This implies that $k=1$ in each case.
Assume that $S$ is isomorphic to one of the simple groups as in the first column of Table~\ref{tbl:s2-M22}. Then by \cite{Atlas}, $G'/M$ has a character $\psi$ of degree as in the third column of the same Table~\ref{tbl:s2-M22}. If $\chi$ is an irreducible constituent of $\psi^{G/M}$, then $\chi(1)=t\psi(1)$, where $t$ divides $|\Out(S)|$. Consequently, $G$ has a character of degree at most $d$ as in the forth column of Table~\ref{tbl:s2-M22}, which is a contradiction.
\begin{table}
  \centering
  \caption{The triples $(S,\psi,d)$ in Lemma~\ref{lem:M22-2}}
  \label{tbl:s2-M22}
  \smallskip
  \begin{tabular}{lccc}
    \hline
    $S$ & $|\Out(S)|$ & $\psi(1)$ & $d$ \\
    \hline
    $A_5$ & $2$ & $5$ & $10$\\
    $A_6$ & $4$ & $5$ & $20$\\
    $A_7$, $L_2(7)$ & $2$ & $6$ & $12$\\
    $L_2(8)$ & $3$ & $8$ & $24$ \\
    \hline
  \end{tabular}
\end{table}
For example, if $S$ is isomorphic to $A_7$ or $L_2(7)$, then $G'/M$ has a character $\psi$ of degree $6$, then $\chi(1)=t\psi(1)=6t$, where $t$ divides $|\Out(S)|$, and so $G$ has a character of degree at most $12$, which is a contradiction. Therefore $G'/M \cong M_{22}$.
\end{proof}

\begin{lemma}\label{lem:M22-3}
If $\theta \in \Irr(M)$, then $I_{G'}(\theta)=G'$ and $M=M'$.
\end{lemma}
\begin{proof}
Suppose  $I:=I_{G'}(\theta)<G'$. By Lemma~\ref{lem:clif} we have $\theta^I= \sum_{i=1}^{k} \phi_i$ where $\phi_i \in \Irr(I)$ for $ i=1,2,...,k$. Let $U/M$ be a maximal subgroup of $G'/M\cong M_{22}$ containing $I/M$ and set $t:=|U:I|$. It follows from Lemma~\ref{lem:clif}(a) that $\phi_i(1)|G':I| \in \cd(G')$, and so $t\phi_i(1)|G':U|$ divides some degrees of $G$. Then $|G':U|$  must divide some character degrees of $G$, and hence by Lemma~\ref{lem:M22.2} one of the following holds: \smallskip

\noindent (i) Suppose $U/M\cong L_3(4)$. Then, for each $i$, $t\phi_i(1)$ divide $7$.
As $U/M \cong L_3(4)$ does not have any subgroup of index $7$, by  ~\cite[p. 23]{Atlas}, so $t=1$ and $I/M\cong U/M\cong L_3(4)$ and $\phi_i(1)$ divide $7$. If $\phi_i(1)=1$, then $\theta$ extend to $\phi_{i}$, and so by Lemma~\ref{lem:clif}(b), $(\phi_{i}\tau )^{G'}\in \Irr(G')$, for all $\tau \in \Irr(I/M)$. If $\tau(1)=64\in \cd(L_3(4))$, then $22\tau(1)=2^{7}\cdot11$ divides some degrees of $G$, which is a contradiction. Hence $\phi_i(1)=7$, for all $i$. Then $\phi_{i_{M}}=e_{i}\theta$, where $e_{i}\neq 1$ is the degree of  a projective representation of $I/M\cong L_3(4)$, and it is impossible by~\cite[p. 24]{Atlas}.\smallskip

\noindent (ii) Suppose  $U/M\cong 2^4:S_5$. Then $I/M\cong 2^4: S_5$ and $\phi_i(1)=1$, for all $i$. Thus $\theta$ extends to $\phi_{i}$ in $I$. It follows from Lemma~\ref{lem:clif}(b) that $\tau(1)|G':I|$ divides some character degrees of $G$, for all $\tau \in I/M$. This is impossible by taking $\tau(1)=4$.\smallskip

\noindent (iii) Suppose  $U/M \cong 2^4:A_6$. In this case, $t\phi_{i}(1)$ divides $2$, $3$ or $5$. It follows from \cite{Atlas-Almost} that $U/M$ has no maximal subgroup of index $2$, $3$ and $5$, and this implies that $I=U$. Therefore, $I/M \cong 2^{4}:A_{6}$, that is to say, $I/M$ has an abelian subgroup $A/M$ of order $2^{4}$ such that $I/A \cong A_{6}$.

Let now $\lambda\in \Irr(A|\theta)$, and write $\lambda^{I}=\sum f_{i}\mu_{i}$. Since $A\unlhd I$, the degree $\mu_{i}(1)$ divides $2$, $3$ or $5$, for all $i$. Since the index of a maximal subgroup of $I/A \cong A_{6}$ is at least $6$, $\lambda$ is $I$-invariant, and so $\mu_{i_A}=f_{i}\lambda$, for all $i$. If $f_{i}=1$ for some $i$, then $\lambda$ extends to $\lambda_{0}\in \Irr(I)$, and so by  Lemma~\ref{lem:clif}(b), $\lambda_{0}\tau$ is an irreducible character of $\lambda^{I}$, for all $\tau \in \Irr(I/A)$, and so $\lambda_{0}(1)\tau(1)=\tau(1)$ divides $2$, $3$, or $5$. This is impossible as we could take $\tau(1)=8 \in \cd(A_{6})$. Therefore, $f_{i}>1$, for all $i$. Moreover, we know from Lemma~\ref{lem:clif}(c) that each $f_{i}$ is the degree of a nontrivial proper irreducible projective representation of $A_{6}$, by \cite[p. 5]{Atlas}, we observe that $f_{i}\in\{3,5\}$. This shows that $\mu(1)/\lambda(1)$ is odd, for all $\mu\in \Irr(I|A)$, and so by Lemma~\ref{lem:sol}, the group $I/A\cong A_{6}$ is solvable, which is a contradiction.

This show $I_{G'}(\theta)=G'$. By Lemma~\ref{lem:schur}, we have that $|M/M'|$ divides the order of Schur multiplier of $G'/M\cong M_{22}$ which is $12$. If $|M/M'|\neq 1$, then $G$ has an irreducible character of degree divisible by one the degrees in the second row of Table~\ref{tbl:c-3-M22}, which is a contradiction. Therefore, $M=M'$.
\end{proof}
\begin{table}
  \centering
  \caption{Some character degrees of $G'/M'$ in Lemma~\ref{lem:M22-3}}
  \label{tbl:c-3-M22}       
  \smallskip
  \begin{tabular}{lccccc}
    \hline
    $G'/M'$ & $2\cdot M_{22}$ & $3\cdot M_{22}$ & $4\cdot M_{22}$ & $6\cdot M_{22}$ & $12\cdot M_{22}$\\
    \hline
    Degree & $440$ & $384$ & $440$ & $384$ & $384$\\
    \hline
  \end{tabular}
\end{table}

\begin{lemma}\label{lem:M22-4}
The subgroup $M$ is trivial, and hence $G'\cong M_{22}$.
\end{lemma}
\begin{proof}
It follows from Lemmas~\ref{lem:M22-2} and \ref{lem:M22-3} that  $G'/M \cong M_{22}$ and $M=M'$. Assume that $M$ is nonabelian, and let $N\leq M$ be a normal subgroup of $G'$ such that $M/N$ is a chief factor of $G'$. Then $M/N\cong S^{k}$ for some nonabelian simple group $S$. By Lemma~\ref{lem:exten}, $S$ has a nontrivial irreducible character $\varphi$ such that $\varphi^{k}\in \Irr(M/N)$ extends to $G'/N$. Now Lemma~\ref{lem:gal}(b) implies that $\varphi(1)^{k}\tau(1)\in \cd(G'/N)\subseteq \cd(G')$, for all $\tau \in \Irr(G'/M)$. As $\varphi(1)>1$, if we choose $\tau\in G'/M$ such that $\tau(1)$ is the largest degree of $M_{22}$, then $\varphi(1)^{k}\tau(1)$ divides no degree of $G$, which is a contradiction. Therefore, $M$ is abelian, and hence we are done.
\end{proof}

\begin{lemma}\label{lem:M22-5}
There exists an abelian group $A$ such that $G/A\cong M_{22}:2$.
\end{lemma}
\begin{proof}
Set $A:=C_G(G')$. Since $G'\cap A=1$ and $G'A\cong G' \times A$, it follows that $G' \cong G'A/A\unlhd G/A\leq \Aut(G')$. Since $G'\cong M_{22}$, we conclude that $G/A$ is isomorphic to $M_{22}$ or $M_{22}:2$. In the case where $G/A$ is isomorphic to $M_{22}$, we conclude that $G\cong A\times M_{22}$. This is impossible as $560\in\cd(G)$ but $M_{22}$ has no character of degree $560$. Therefore, $G/A$ is isomorphic to $M_{22}:2$.
\end{proof}


\end{document}